\documentclass[aps,11pt,pre,byrevtex,bibnotes,superscriptaddress,nofootinbib,showpacs,notitlepage,showkeys,longbibliography]{revtex4-1}

\usepackage{microtype}
\usepackage{graphicx}
\usepackage{amsmath}
\usepackage{amssymb}
\usepackage{amsthm}
\usepackage{mathrsfs} 
\usepackage{color}
\definecolor{myblue}{rgb}{0.153,0.322,0.706}
\usepackage[colorlinks,linkcolor=myblue,urlcolor=myblue,citecolor=myblue,breaklinks=true]{hyperref}

\newcommand{\be}{\begin{equation}}
\newcommand{\ee}{\end{equation}}

\newcommand{\tG}{\tilde G}
\newcommand{\tQ}{\tilde Q}
\newcommand{\bb}{\bar b}
\newcommand{\eqdist}{\overset{d}{=}}

\numberwithin{equation}{section}

\newtheorem{theorem}{Theorem}[section]
\newtheorem{lemma}[theorem]{Lemma}

\newtheorem{remark}[theorem]{Remark}

\newcommand{\R}{\mathbb{R}}

\newcommand{\N}{\mathbb{N}}

\newcommand{\PP}{\mathbb{P}}
\newcommand{\EE}{\mathbb{E}}

\newcommand{\eee}{\mathrm{e}}
\newcommand{\ddd}{\mathrm{d}}

\newcommand{\Ai}{\operatorname{Ai}}

\newcommand{\AI}{\operatorname{AI}}

\DeclareMathOperator{\erf}{erf}
\newcommand{\Var}{\mathbb{V}\mathrm{ar}}

\begin{document}

%%%%%%%%%%%%%%%%%%%%%%%%%%%%%%%%%%%%%%%%%%

\title{Properties of additive functionals of Brownian motion with resetting}

\author{Frank den Hollander}
\email{denholla@math.leidenuniv.nl}
\affiliation{Mathematical Institute, Leiden University, 2300 RA Leiden, The Netherlands}

\author{Satya N. Majumdar}
\email{satya.majumdar@u-psud.fr} 
\affiliation{Laboratoire de Physique Statistique et Mod\`eles Statistiques, 
UMR 8626, Universit\'e Paris-Sud, Orsay 91405, France}

\author{Janusz M. Meylahn}
\email{j.m.meylahn@math.leidenuniv.nl}
\affiliation{Mathematical Institute, Leiden University, 2300 RA Leiden, The Netherlands}

\author{Hugo Touchette}
\email{htouchet@alum.mit.edu, htouchette@sun.ac.za}
\affiliation{National Institute for Theoretical Physics (NITheP), Stellenbosch 7600, South Africa}
\affiliation{Institute of Theoretical Physics, Department of Physics, Stellenbosch University, 
Stellenbosch 7600, South Africa}

\date{\today}

%%%%%%%%%%%%%%%%%%%%%%%%%%%%%%%%%%%%%%%%%%%%
\begin{abstract}
We study the distribution of additive functionals of reset Brownian motion, a variation of 
normal Brownian motion in which the path is interrupted at a given rate and placed back 
to a given reset position. Our goal is two-fold: (1) For general functionals, we derive a large 
deviation principle in the presence of resetting and identify the large deviation rate function 
in terms of a variational formula involving large deviation rate functions without resetting. 
(2) For three examples of functionals (positive occupation time, area and absolute area), we 
investigate the effect of resetting by computing distributions and moments, using a 
formula that links the generating function with resetting to the generating function 
without resetting. 
\end{abstract}


\maketitle

%%%%%%%%%%%%%%%%%%%%%%%%%%%%%%%%%%%%%%%%%%%
\section{Introduction}

In this paper we study a variation of Brownian motion (BM) that includes resetting events 
at random times. Let $(W_t)_{t \geq 0}$ be a BM on $\R$ and consider a Poisson process 
on $[0,\infty)$ with intensity $r \in (0,\infty)$ and law $\PP$, producing $N(T)$ random 
points $\{\sigma_{i}\}_{i=1}^{N(T)}$ in the time interval $[0, T]$, satisfying $\EE[N(T)] = rT$. 
From these two processes, we construct the \emph{reset Brownian motion} (rBM),
$(W^r_t)_{t \geq 0}$, by `pasting together' $N(T)$ independent trajectories of the BM, 
all starting from a reset position $x_* \in \R$ and evolving freely over the successive time 
lapses of length $\tau_i$ with
\begin{equation}
\tau_i = \sigma_{i+1} - \sigma_{i}, \qquad i=0,\ldots,N(T)-1,
\end{equation}
with $\sigma_0=0$. More precisely, $W^r_t = x_* + W^i_t$ for $t \in [\sigma_{i},\sigma_{i+1})$ with 
$(W^i_t)_{t \geq 0}$, $i=0,\ldots,N(T)-1$, independent BMs starting at $0$. Without loss of 
generality, we assume that $x_*=0$. We denote by $\PP_{r}$ the probability with respect to 
rBM with reset rate $r$.

The properties of rBM, and reset processes in general \cite{montero2017}, have been the 
subject of several recent studies, related to random searches and randomized algorithms 
\cite{Evans11, Evans13, Kusmierz14,luby1993,Avrachenkov13,Avrachenkov17,Chechkin18,Belan18,Bodrova19} (which can be made more efficient by the 
addition of resetting \cite{Evans11a}), queueing theory (where resetting accounts for the 
accidental clearing of queues or buffers), as well as birth-death processes 
\cite{Pakes78, Brockwell85, Kyriakidis94, Pakes97, manrubia1999, Dharmaraja15} (in which 
a population is drastically reduced as a result of natural disasters or catastrophes). In 
biology, the attachment, targeting and transcription dynamics of enzymes, proteins and 
other bio-molecules can also be modelled with reset processes 
\cite{Benichou07, Harris17,visco2010,Meylahn15,reuveni2016,roldan2016,pal2017}. 

Resetting has the effect of creating a `confinement' around the reset position, which can bring 
the process from being non-stationary to being stationary. The simplest example is rBM, which 
has a stationary density $\rho$ given by \cite{Evans11}
\begin{equation}
\rho(x) = \sqrt{\frac{r}{2}}\, e^{-\sqrt{2r} |x|}, \quad x\in\R.
\end{equation}
The motivation for the present paper is to study the effect of the confinement on the distribution 
of additive functionals of rBM of the general form
\begin{equation}
\label{eq:FTdef}
F_T = \int_0^T f(W_t^r) \, \ddd t,
\end{equation}
where $f$ is a given $\R$-valued measurable function. We are especially interested in studying the 
effect of resetting on the large deviation properties of these functionals, and to determine whether 
resetting is `strong enough' to bring about a large deviation principle (LDP) for the sequence of 
random variables $(T^{-1}F_T)_{T>0}$ when it does not satisfy the LDP without resetting.

For this purpose, we use a recent result \cite{Meylahn15,Meylahn15a} based on the renewal structure 
of reset processes that links the Laplace transform of the Feynman-Kac generating function of $F_T$ 
with resetting to the same generating function without resetting. Additionally, we derive a variational formula 
for the large deviation rate function of $(T^{-1}F_T)_{T>0}$, obtained by combining the LDPs for the 
frequency of resets, the duration of the reset periods, and the value of $F_T$ in between resets. This 
variational formula complements the result based on generating functions by providing insight into 
how a large deviation event is created in terms of the constituent processes. These two results are 
stated in Secs.~\ref{sec:res}--\ref{sec:genratepropproof} and, in principle, apply to any functional $F_T$ 
of the type defined in \eqref{eq:FTdef}. We illustrate them for three particular functionals: 
\begin{equation}
\label{eq:ABC}
A_T = \int_0^T 1_{[0,\infty)}(W^r_t)\,\ddd t, \qquad
B_T = \int_0^T W^r_t\,\ddd t, \qquad
C_T = \int_0^T |W^r_t|\,\ddd t,
\end{equation}
i.e., the positive occupation time, the area and the absolute area (the latter can also be interpreted as 
the area of rBM reflected at the origin). These functionals are discussed in Secs.~\ref{sec:occtime}, 
\ref{sec:area} and \ref{sec:absolutearea}, respectively.

It seems possible to extend part of our results to general diffusion processes with resetting, 
although we will not attempt to do so in this paper. The advantage of focusing on rBM is 
that we can obtain exact results.

%%%%%%%%%%%%%%%%%%%%%%%%%%%%%%%%%%%%%%%%%%
\section{Two theorems}
\label{sec:res}

In this section we present two theorems that will be used to study distributions (Theorem~\ref{thm:RR}) 
and large deviations (Theorem~\ref{thm:var}) associated with additive functionals of rBM. 

The first result is based on the generating function of $F_T$: 
\begin{equation}
\label{eq:genfunc}
G_{r}(k, T) = \EE_{r}\big[\eee ^{kF_{T}}\big], \qquad k \in \R,\,T \in [0,\infty),
\end{equation}
where $\EE_{r}$ denotes the expectation with respect to rBM with rate $r$. The Laplace transform \cite{Widder41}
of this function is defined as 
\begin{equation}
\label{eq:lapgenfunc}
\tG_{r}(k, s) = \int_{0}^{\infty} \ddd T\,\eee ^{-sT}\,G_{r}( k,T),\qquad k \in \R,\, s \in [0,\infty).
\end{equation}
Both may be infinite for certain ranges of the variables.  The same quantities are defined analogously 
for the reset-free process and are given the subscript $0$. The following theorem expresses the
reset Laplace transform in terms of the reset-free Laplace transform. 

\begin{theorem}
\label{thm:RR}
If $k \in \R$ and $s \in [0,\infty)$ are such that $r\tG_{0}(k, r+s)<1$, then
\begin{equation}
\label{eq:RR}
\tG_{r}(k, s) = \frac{\tG_{0}(k, r+s)}{1-r \tG_{0}(k, r+s)}.
\end{equation}
\end{theorem}

\begin{proof}
Theorem~\ref{thm:RR} was proved in \cite{Meylahn15} with the help of a renewal argument 
relating the process with resetting to the one without resetting. For completeness we write
out the proof. For fixed $T$, split according to whether the first reset takes place at $0 < t \leq T$ 
or $t>T$:  
\begin{equation}
\EE_{r}[\eee^{kF_T}\big] 
= \int_0^T \ddd t\,r\eee^{-rt}\,\EE_{0}\big[\eee^{kF_t}\big]\,\EE_{r}\big[\eee^{kF_{T-t}}\big]
+ \int_T^\infty \ddd t\,r\eee^{-rt}\,\EE_{0}\big[\eee^{kF_T}\big].
\end{equation}
Substitute this relation into \eqref{eq:genfunc} and afterwards into \eqref{eq:lapgenfunc}, and 
interchange the integration over $T$ and $t$, to get
\begin{equation}
\begin{aligned}
\tG_{r}(k, s) &= \int_0^\infty \ddd t\,r\eee^{-rt}\,\EE_{0}\big[\eee^{kF_t}\big]\,
\eee^{-st} \int_t^\infty \ddd T\,\eee^{-s(T-t)}\,\EE_{r}\big[\eee^{kF_{T-t}}\big]
+ \int_0^\infty \ddd T\,\eee^{-rT}\,\eee^{-sT}\,\EE_{0}\big[\eee^{kF_T}\big]\\
&=r\left(\int_0^\infty \ddd t\,\eee^{-(r+s)t}\,\EE_{0}\big[\eee^{kF_t}\big]\right)
\left(\int_0^\infty \ddd T'\,\eee^{-sT'}\,\EE_{r}\big[\eee^{kF_{T'}}\big]\right)
+ \int_0^\infty \ddd T\,\eee^{-(r+s)T}\,\EE_{0}\big[\eee^{kF_T}\big]\\
&= r \tG_{0}(k,r+s)\tG_{r}(k,s) + \tG_{0}(k, r+s).
\end{aligned} 
\end{equation}
Solving for $\tG_{r}(k, s)$, we get \eqref{eq:RR}. 
\end{proof}

As shown in \cite{Meylahn15}, Theorem~\ref{thm:RR} can be used to study the effect of resetting 
on the distribution of $F_T$. In particular, if the dominant singularity of $\tG_{r}(k, s)$ is a single 
pole, then Theorem~\ref{thm:RR} can be used to get the LDP with resetting, under the assumption 
that 
\begin{equation}
\label{eq:assexist}
\forall\,T>0\colon \qquad G_{0}(k,T) \text{ exists for $k$ in an open neighbourhood of $0$ in $\R$.}
\end{equation}
In Theorem~\ref{thm:var} below we show that, for every $r>0$, $(T^{-1}F_T)_{T>0}$ satisfies 
the LDP on $\R$ with speed $T$. Informally, this means that
\begin{equation}
\forall\,\phi \in \R\colon \qquad \frac{\PP_{r}(T^{-1}F_{T} \in \ddd \phi)}{\ddd \phi} 
= \eee^{-T\chi_r(\phi) + o(T)}, \qquad T \to \infty,
\end{equation}
where $\chi_r\colon\,\R \to [0,\infty)$ is the rate function. See Appendix~\ref{app:LDP} for the formal 
definition of the LDP. 

Theorem~\ref{thm:var} below provides a variational formula for $\chi_r$ in terms of the rate 
functions of the three constituent processes underlying $F_T$, namely (see \cite[Chapters I-II]{H00}):  
\begin{itemize}
\item[(1)]
The rate function for $(T^{-1}N(T))_{T>0}$, the number of resets per unit of time:
\begin{equation}
\label{ICramer}
I_{r}(n) = n\log\Big(\frac{n}{r}\Big) - n + r, \qquad n \in [0, \infty).
\end{equation}
\item[(2)]
The rate function for $(N^{-1} \sum_{i=1}^N \delta_{\tau_i})_{N\in\mathbb{N}}$, the empirical 
distribution of the duration of the reset periods:
\begin{equation}
\label{JSanov}
J_{r}(\mu) = h(\mu \mid \mathcal{E}_r), \qquad \mu \in \mathcal{P}([0,\infty)).
\end{equation}
Here, $\mathcal{P}([0, \infty))$ is the set of probability distributions on $[0,\infty)$, 
$\mathcal{E}_r$ is the exponential distribution with mean $1/r$, and $h(\cdot \mid \cdot)$
denotes the relative entropy
\begin{equation}
h(\mu \mid \nu) = \int_{0}^{\infty} \mu(dx) \log\left[\frac{\ddd\mu}{\ddd\nu}(x)\right],
\qquad \mu,\nu \in \mathcal{P}([0, \infty)).
\end{equation}
\item[(3)]
The rate function for $(N^{-1}\sum_{i=1}^N F_{\tau,i})_{N \in \mathbb{N}}$, the empirical 
average of i.i.d.\ copies of the \emph{reset-free} functional $F_\tau$ over a time $\tau$: 
\begin{equation}
\label{KCramer}
K_\tau(u) = \sup_{v \in \R}\, \{uv-M_\tau(v)\},
\qquad u \in \R,\,\tau \in [0,\infty).
\end{equation}
Here, $M_\tau(v) = \log \EE_0\big[\eee^{vF_\tau}\big]$ is the cumulant generating function of 
$F_\tau$ without reset and we require, for all $\tau \in [0,\infty)$, that $M_\tau$ exists in an open 
neighbourhood of $0$ in $\R$ (which is equivalent to \eqref{eq:assexist}). It is known that 
$K_\tau$ is smooth and strictly convex on the interior of its domain (see \cite[Chapter I]{H00}).
\end{itemize}

\begin{theorem}
\label{thm:var}
For every $r>0$, the family $(\PP_{r}(T^{-1}F_{T} \in \cdot\,))_{T>0}$ satisfies the LDP on $R$ 
with speed $T$ and with rate function 
$\chi_r$ given by
\begin{equation}
\label{eq:VarRate}
\chi_r(\phi) = \inf_{(n,\mu,w) \in \Phi(\phi)} 
\Big\{I_{r}(n) + n J_{r}(\mu) + n \int_{0}^{\infty} \mu(\ddd t)\,K_{t}(w(t)) \Big\},
\qquad \phi \in \R,
\end{equation}
where 
\begin{equation}
\Phi(\phi) = \left\{(n,\mu,w) \in [0, \infty) \times \mathcal{P}([0, \infty)) \times \mathcal{B}([0,\infty);\R)\colon\, 
n \int_{0}^{\infty} \mu(\ddd t)\,w(t) = \phi \right\} 
\end{equation}
with $\mathcal{B}([0,\infty);\R)$ the set of Borel-measurable functions from $[0,\infty)$ to $\R$.
\end{theorem}

\begin{proof}
The LDP for $(T^{-1}F_T)_{T>0}$ follows by combining the LDPs for the constituent processes 
and using the contraction principle \cite[Chapter III]{H00}. The argument that follows is 
informal. However, the technical details are standard and are easy to fill in. 

First, recall that $N(T)$ is the number of reset events in the time interval $[0,T]$. By Cram\'er's 
Theorem \cite[Chapter I]{H00}, $(T^{-1}N(T))_{T>0}$ satisfies the LDP on $[0,\infty)$ with speed 
$T$ and with rate function $I_{r}$ in \eqref{ICramer}, because resetting occurs according to a 
Poisson process with intensity $r$. This rate function has a unique zero at $n=r$ and takes the 
value $r$ at $n=0$. 

Next, consider the empirical distribution of the reset periods,
\begin{equation}
\mathcal{L}_{m} = \frac{1}{m}\sum_{i=1}^{m} \delta_{\tau_{i}}.
\end{equation}
By Sanov's Theorem \cite[Chapter II]{H00}, $(\mathcal{L}_{m})_{m\in\mathbb{N}}$ satisfies the LDP 
on $\mathcal{P}([0,\infty))$, the space of probability distributions on $[0,\infty)$, with speed $m$ 
and with rate function $J_{r}$ in \eqref{JSanov}. This rate function has a unique zero at $\mu
=\mathcal{E}_r$. 

Finally, consider the empirical average of $N$ independent trials 
$\{F_{\tau,i}\}_{i=1}^N$ of the reset-free process of length $\tau$, 
\begin{equation}
m_N  = \frac{1}{N} \sum_{i=1}^{N} F_{\tau,i}.
\end{equation}
By Cram\'er's Theorem, $(m_N)_{N\in\mathbb{N}}$ satisfies the LDP on $[0, \infty)$ with 
speed $N$ and with rate function $K_{\tau}$ in \eqref{KCramer}. This rate function has a 
unique zero at $u=\EE_0(F_\tau)$.

Now, the probability that $nt\,\mu(\ddd \tau)$ excursion times of length $\tau$ contribute 
an amount $u\,nt\,\mu(\ddd \tau)$ to the integral equals
\begin{equation}
\eee^{-nt\,\mu(\ddd\tau)\,K_{\tau}(u) + o(nt)}
\end{equation}
for any $u \in \R$. If we condition on $N(T) = nT$ and $\mathcal{L}_{N(T)} = \mu$, and pick 
$w\in\mathcal{B}([0,\infty);\R)$, then the probability that $nT$ duration times contribute an 
amount $\phi\,nT$ to the integral, with
\begin{equation}
\phi = n \int_{0}^{\infty} \mu(\ddd t)\,w(t),
\end{equation}
equals
\begin{equation}
\eee^{-nT\int_{0}^{\infty}\mu(\ddd t)K_{t}(w(t)) + o(nT)}.
\end{equation}
Therefore, by the contraction principle \cite[Chapter III]{H00}, 
\begin{equation}
\frac{\PP_{r}(T^{-1}F_{T} \in \ddd \phi)}{\ddd \phi} = \mathrm{e}^{-T \chi_r(\phi)+o(T)},
\end{equation}
where $\chi_r(\phi)$ is given the variational formula in \eqref{eq:VarRate}. 
\end{proof}

\begin{remark}
{\rm A priori, Theorem~\ref{thm:var} is to be read as a \emph{weak} LDP: the level sets of 
$\chi_r$ need not be compact, e.g.\ it is possible that $\chi_r \equiv 0$. Under additional 
assumptions, $\chi_r$ has compact level sets, in which case Theorem~\ref{thm:var} can 
be read as a \emph{strong} LDP. See Appendix~\ref{app:LDP} for more details.}  
\end{remark}

We will see that the three functionals in \eqref{eq:ABC} have rate functions of different type,
namely, $\chi_r$ is:
\begin{itemize}
\item[$A_T$:] zero at $\tfrac12$, strictly positive and finite on $[0,1]\setminus\{\tfrac12\}$,
infinite on $\R\setminus [0,1]$ (strong LDP).
\item[$B_T$:] zero on $\R$ (weak LDP).
\item[$C_T$:] zero on $[1/\sqrt{2r},\infty)$, strictly positive and finite on $(0,1/\sqrt{2r})$,
infinite on $(-\infty,0]$ (strong LDP).
\end{itemize}

%%%%%%%%%%%%%%%%%%%%%%%%%%%%%%%%%

\section{Two properties of the rate function}
\label{sec:genratepropproof}

The variational formula in (\ref{eq:VarRate}) can be used to derive some general properties of 
the rate function with resetting. In this section, we show that the rate function is flat beyond 
the mean with resetting \emph{provided the mean without resetting diverges}, and is quadratic 
below and near the mean with resetting. Both properties will be illustrated in Sec.~\ref{sec:absolutearea} 
for the absolute area of rBM. 

\subsection{Zero rate function above the mean}
\label{sec:rfflat}

For the following theorem, we define
\begin{equation}
\phi^{*}_r = \lim_{T\to\infty} \mathbb{E}_r[T^{-1}F_T], \qquad r \geq 0.
\end{equation}
Moreover, we must assume that $f\geq 0$ in \eqref{eq:FTdef}, and that there exists a $C \in (0,\infty)$ 
such that 
\begin{align}
\label{eq:keyassumption}
\mathbb{E}[f(W_{t})^{2}] \leq C\,\mathbb{E}[f(W_{t})]^{2}\quad \forall t\geq 0. 
\end{align}

\begin{remark}
Assumption \eqref{eq:keyassumption} holds for $f(x) = |x|^{\gamma}$, $x \in \R$, and 
any $\gamma \in [0,\infty)$, and for $f(x)=1_{[0,\infty)}(x)$, $x \in \R$.
\end{remark}

\begin{theorem}
\label{thm:genrateprop}
Suppose that $f$ satisfies \eqref{eq:keyassumption} and that $\phi^{*}_0 = \infty$. 
For every $r>0$, if $\phi^{*}_r < \infty$, then
\begin{equation}
\chi_r(\phi) = 0 \quad \forall\, \phi \geq \phi^{*}_r. 
\end{equation}
\end{theorem}

In order to prove the theorem we need the following.

\begin{lemma}
\label{lemma:keyassumption}
If \eqref{eq:keyassumption} holds, then the following \emph{zero-one law} applies:
\begin{align}
\mathbb{P}\Big(\lim_{T\rightarrow\infty}T^{-1}F_{T} = \infty\Big) = 1
\quad \Longleftrightarrow \quad \phi^{*}_0 = \infty.
\end{align}
\end{lemma}
\begin{proof}
Because $(W_{t})_{t\geq 0}$ has a \emph{trivial} tail sigma-field, we have 
\begin{align}
\mathbb{P}\Big(\lim_{T\rightarrow \infty}T^{-1}F_{T}=\infty\Big) \in \{0, 1\}. 
\end{align}
It suffices to exclude that the probability is 0. First note that \eqref{eq:keyassumption} implies
\begin{align}
\label{eq:keyassumptionimplication}
\mathbb{E}[(T^{-1}F_{T})^{2}] \leq C\,\mathbb{E}[T^{-1}F_{T}]^{2} \qquad \forall\,T > 0. 
\end{align}
Indeed, 
\begin{align}
T^{2}\mathbb{E}[(T^{-1}F_{T})^{2}] 
&= \int_{0}^{T} \ddd s \int_{0}^{T} \ddd t\,\,\mathbb{E}[f(W_{s})f(W_{t})]\nonumber\\
&\leq \int_{0}^{T} \ddd s \int_{0}^{T} \ddd t\,\,\sqrt{\mathbb{E}[f(W_{s})^{2}]\,\mathbb{E}[f(W_{t})^{2}]}\nonumber\\
&\leq C \int_{0}^{T}\ddd s \int_{0}^{T}\ddd t\,\,\mathbb{E}[f(W_{s})]\,\mathbb{E}[f(W_{t})]\nonumber\\
&= C\,T^{2}\,\mathbb{E}[T^{-1}F_{T}]^{2},
\end{align}
where the first inequality uses Cauchy--Schwarz and the second inequality uses \eqref{eq:keyassumption}. 
Armed with \eqref{eq:keyassumptionimplication}, we can use the Paley--Zygmund inequality 
\begin{align}
\mathbb{P}(T^{-1}F_{T}\geq \delta\mathbb{E}[T^{-1}F_{T}]) \geq (1-\delta)^{2}\frac{\mathbb{E}
[T^{-1}F_{T}]^{2}}{\mathbb{E}[(T^{-1}F_{T})^{2}]} \qquad \forall\,\delta \in (0, 1)\, \forall\,T>0,
\end{align}
to obtain
\begin{align}
\mathbb{P}\Big(\frac{T^{-1}F_{T}}{\mathbb{E}[T^{-1}F_{T}]}\geq \delta\Big)\geq (1-\delta)^{2}
\frac{1}{C} \qquad \forall\, \delta \in (0, 1)\,\forall\, T>0.
\end{align}
Hence if $\lim_{T\rightarrow\infty}\mathbb{E}[T^{-1}F_{T}] = \infty$, then
\begin{align}
\mathbb{P}\Big(\lim_{T\rightarrow\infty}T^{-1}F_{T}=\infty\Big)\geq (1-\delta)^{2}\frac{1}{C}>0
\qquad \forall\, \delta \in (0,1),
\end{align}
which completes the proof.
\end{proof}

We now turn to proving Theorem~\ref{thm:genrateprop}. Again, the argument that follows is
informal, but the technical details are standard.

\begin{proof}[Proof of Theorem~\ref{thm:genrateprop}]
The variational formula for the rate function in \eqref{eq:VarRate} is a constrained functional 
optimization problem that can be solved using the method 
of Lagrange multipliers. For fixed $n$ and $\mu$, the Lagrangian reads
\begin{equation}
\mathcal{L}(w(\cdot)) =  I_{r}(n) + n J_{r}(\mu) + n \int_{0}^{\infty}\mu(\ddd t)\,K_{t}(w(t)) 
-\lambda n\int_{0}^{\infty}\mu(\ddd t)\, w(t),
\end{equation}
where $\lambda$ is the Lagrange multiplier that enforces the constraint 
\begin{align}
\label{eq:conK}
n \int_{0}^{\infty}\mu(\ddd t) w(t) = \phi. 
\end{align} 
We look for solutions $w_{\lambda}(\cdot)$ of the equation $\frac{\partial \mathcal{L}}{\partial w(t)}
(\cdot)=0$ for all $t \geq 0$, i.e.,
\begin{align}
\label{eq:Kprime}
K_{t}'(w_{\lambda}(t)) = \lambda, \qquad t \geq 0, 
\end{align} 
where $w_{\lambda}(\cdot)$ must satisfy the constraint $n \int_{0}^{\infty} \mu(\ddd t)\,w_\lambda(t) 
= \phi$. To that end, let $L_{t}(\cdot)$ be the inverse of $K_{t}'(\cdot)$, i.e.,
\begin{equation}
\label{eq:Ldef}
K_{t}'(L_t(\lambda)) = \lambda, \qquad \lambda \in \R,\,t>0.
\end{equation}
Then \eqref{eq:Kprime} becomes
\begin{equation}
w_{\lambda}(t) = L_{t}(\lambda), \qquad t \geq 0,
\end{equation} 
and so
\begin{equation}
\label{eq:rfreform}
\chi_{r}(\phi) = \inf_{n \in [0,\infty),\, \mu \in \mathcal{P}([0,\infty))} \Big\{I_{r}(n) + nJ_{r}(\mu) 
+ n\int_{0}^{\infty}\mu(\ddd t)\,K_{t}(L_{t}(\lambda)) \Big\},
\end{equation}
where $\lambda=\lambda(n,\mu)$ must be chosen such that
\begin{align}
\label{eq:conL}
n \int_{0}^{\infty}\mu(\ddd t)\,L_{t}(\lambda) = \phi.
\end{align}

Our task is to show that $\chi_r$ is zero on $[\phi^*_r,\infty)$ when $\phi^*_0=\infty$.
To do so, we perturb $\chi_{r}(\phi)$ around $\phi^*_r$. To see how, we first rescale time. 
The proper rescaling depends on how $F_T$ scales with $T$ without resetting. For the 
sake of exposition, we first consider the case where there exists an $\alpha \in (1,\infty)$ 
such that 
\begin{equation}
\label{eqdis1}
T^{-\alpha} F_T \eqdist F_1 \qquad \forall\,T>0,
\end{equation}
where $\eqdist$ means equality in distribution. For example, for the area and the absolute 
area we have $\alpha=\tfrac32$, while for the positive occupation time we have $\alpha=1$. 
(Note, however, that neither the area nor the positive occupation time qualify for the theorem 
because $\phi^*_0=0$, respectively, $\phi^*_0=\tfrac12$.)  Afterwards we will explain how to 
deal with the general case. 

By \eqref{KCramer}, \eqref{eq:Ldef} and \eqref{eqdis1}, we have
\begin{equation}
\label{eq:scal}
K_{t}(u) = K_1(u t^{-\alpha}), \quad u \in \R, \, t > 0, \qquad L_t(\lambda) = L_1(\lambda t^\alpha)\,t^\alpha,
\quad \lambda \in \R,\,t>0.
\end{equation}
The rescaling in \eqref{eq:scal} changes the integral in \eqref{eq:rfreform} to
\begin{equation}
n \int_{0}^{\infty} \mu(\ddd t)\,K_1(L_1(\lambda t^\alpha))
\end{equation}
and the constraint in \eqref{eq:conL} to
\begin{equation}
n \int_{0}^{\infty}\mu(\ddd t)\,L_1(\lambda t^\alpha)\,t^{\alpha} = \phi.
\end{equation}

We claim that, for every $n \in (0,\infty)$, we can find a minimising sequence of probability distributions 
$(\mu_{m})_{m\in\N}$ (depending on $n$) such that $\lambda=\lambda(n,\mu_{m})=0$ for all 
$m\in\N$ and such that $\mu_{m}$ converges as $m\to\infty$ to $\mathcal{E}_r$ pointwise 
and in the $L^{1}$-norm, but not in the $L^{\alpha}$-norm. We will show that this implies that 
$\chi_r(\phi)=0$ for $\phi >\phi^{*}_r$. We will construct the sequence $(\mu_{m})_{m\in\N}$ 
by perturbing $\mathcal{E}_r$ slightly, adding a small probability mass near some large time 
and taking the same probability mass away near time $0$. 

%%%%%%%%%%%%%%%%%%%%%%%%%%%%%%%%%%%%%%
\begin{figure}[htbp]
\begin{center}
\includegraphics{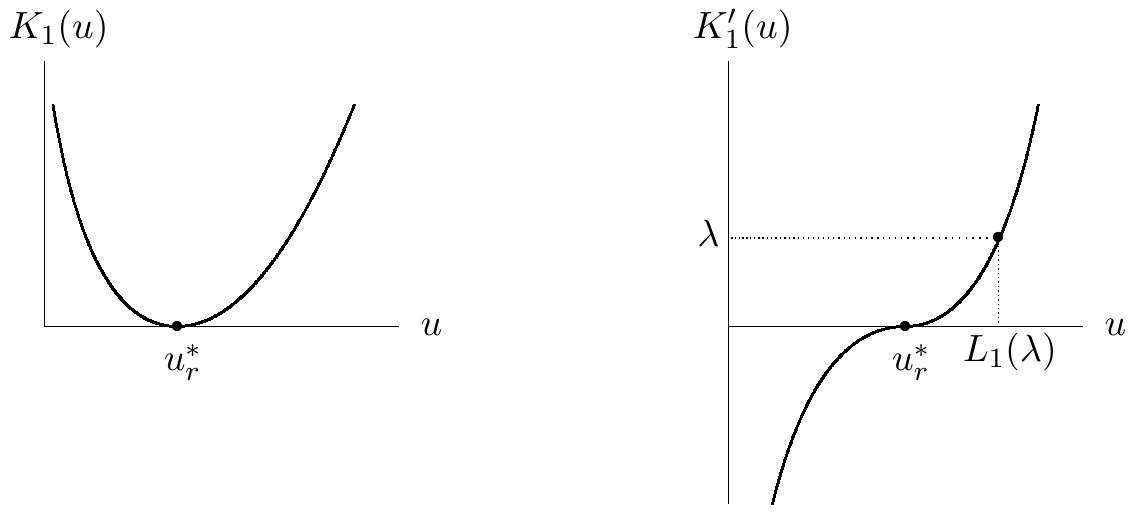}
\caption{Qualitative plot of $u \mapsto K_1(u)$ and $u \mapsto K'_1(u)$ on $\R$. The domain of $K_1$ is a subset 
of $\R$. In the interior of this domain, $K_1$ is smooth and strictly convex.}
\label{fig:KL}
\end{center}
\end{figure}
%%%%%%%%%%%%%%%%%%%%%%%%%%%%%%%%%%%%%

Let $u_r^{*}$ be such that $K_1(u_r^*) = 0$, i.e.,
\begin{equation}
\label{intc1}
r \int_{0}^{\infty} \mathcal{E}_r(\ddd t)\,u_r^{*}\,t^{\alpha} = \phi^*_r
\end{equation}
(see Fig.~\ref{fig:KL}; recall that $\mathcal{E}_r(\ddd t) = r\mathrm{e}^{-rt}\,\ddd t$). Since $u_r^* = L_1(0)$, 
if we require the probability distribution $\mu$ over which we minimise to satisfy
\begin{equation}
\label{intc2}
n \int_{0}^{\infty} \mu(\ddd t)\,u_r^{*}\,t^{\alpha} = \phi, 
\end{equation}
then the scaled version of the optimisation problem in \eqref{eq:rfreform} reduces to 
\begin{equation}
\label{inf}
\inf_{n \in [0,\infty)} \Big\{ I_{r}(n) + n \inf_{\mu \in \mathcal{P}([0,\infty))} J_{r}(\mu) \Big\}. 
\end{equation}
Our goal is to prove that this infimum is zero for all $\phi > \phi^*_r$ when $\phi^*_0=\infty$.

We get an upper bound by picking $n=r$ and
\begin{equation}
\mu_{m}(\ddd t) = \mathcal{E}_r(\ddd t) + \nu_{m}(\ddd t)
\end{equation}
with 
\begin{equation}
\nu_{m}(\ddd t) = -\epsilon_{m}\delta_{0}(\ddd t) + \epsilon_{m}\delta_{\theta_{m}}(\ddd t),
\end{equation}
where $\epsilon_{m},\theta_{m}$ will be chosen later such that $\lim_{m\to\infty} \epsilon_{m} 
= 0$ and $\lim_{m\to\infty} \theta_{m} = \infty$. Substituting this perturbation into \eqref{intc2} 
and using \eqref{intc1}, we get
\begin{equation}
\label{intc}
r u^{*}_r\,\epsilon_m (\theta_m)^{\alpha} = \phi - \phi^*_r, 
\end{equation}
which places a constraint on our choice of $\epsilon_m, \theta_m$. On the other hand, substituting 
the perturbation into the expression for $J_{r}(\mu)$, we obtain
\begin{equation}
\begin{aligned}
\label{eq:Jrewalt}
J_{r}(\mu_m) 
&= \int_{0}^{\infty} (\mathcal{E}_r - \epsilon_{m} \delta_{0} + \epsilon_{m}\delta_{\theta})(\ddd t)\, 
\log\Big(\frac{\mathcal{E}_r - \epsilon_{m} \delta_{0} 
+ \epsilon_{m}\delta_{\theta}}{\mathcal{E}_r}\Big)(t)\\
&= \int_{0}^{\infty} \mathcal{E}_r(\ddd t)\, \log\Big(1 + \frac{- \epsilon_{m} \delta_{0} 
+ \epsilon_{m}\delta_{\theta}}{\mathcal{E}_r}\Big)(t)\\
&\qquad - \epsilon_{m} \int_{0}^{\infty} \delta_{0}(\ddd t)\, 
\log\Big(1 + \frac{- \epsilon_{m} \delta_{0} 
+ \epsilon_{m}\delta_{\theta}}{\mathcal{E}_r}\Big)(t)\\
&\qquad + \epsilon_{m} \int_{0}^{\infty} \delta_{\theta_{m}}(\ddd t)\, 
\log\Big(1 + \frac{- \epsilon_{m} \delta_{0} 
+ \epsilon_{m}\delta_{\theta}}{\mathcal{E}_r}\Big)(t).
\end{aligned}
\end{equation}
For a proper computation, $\delta_0$ and $\delta_\theta$ must be approximated by $\eta^{-1}\,1_{[0,\eta]}$ 
and $\eta^{-1}\,1_{[\theta,\theta+\eta]}$, followed by $\eta \downarrow 0$. Doing so, after we perform the 
integrals, we see that the three terms in the right-hand side of \eqref{eq:Jrewalt} become
\begin{equation}
\begin{aligned}
&r\eta \log\Big( \frac{1-\epsilon_{m}/\eta}{r}\Big) 
+ r\mathrm{e}^{-r\theta_{m}}\eta 
\log\Big( 1 + \frac{\epsilon_{m}/\eta}{r\mathrm{e}^{-r\theta_{m}}}\Big),\\
&-\epsilon_{m}\log\Big(1 -\frac{\epsilon_{m}/\eta}{r}\Big),\\
&+\epsilon_{m}\log\Big(1+\frac{\epsilon_{m}/\eta}{r\mathrm{e}^{-r\theta_{m}}}\Big).
\end{aligned}
\end{equation}
For all of these terms to vanish as $m \to \infty$ followed by $\eta \downarrow 0$, it suffices to pick 
$\epsilon_{m}$ and $\theta_{m}$ such that $\lim_{m\to\infty} \epsilon_{m} = 0$, $\lim_{m\to\infty} 
\theta_{m} = \infty$ and $\lim_{m\to \infty} \theta_{m}\epsilon_{m} = 0$. Clearly, this can be done
while matching the constraint in \eqref{intc} for any $\phi>\phi^*_r$, because $\alpha \in (1,\infty)$, 
and so we conclude that indeed the infimum in \eqref{inf} is zero. 

It is easy to check that the same argument works when, instead of \eqref{eqdis1}, there exists
a $T \mapsto L(T)$ with $\lim_{T \to \infty} L(T) = \infty$ such that
\begin{equation}
\label{eqdis2}
(TL(T))^{-1} F_T \eqdist F_1 \qquad \forall\,T>0.
\end{equation}
Indeed, then the constraint in \eqref{intc1} becomes $r u^{*}_r\,\epsilon_m \theta_m L(\theta_m) 
= \phi - \phi^*_r$, which can be matched too. It is also not necessary that the scaling in
\eqref{eqdis1} and \eqref{eqdis2} hold for all $T>0$. It clearly suffices that they hold 
asymptotically as $T \to\infty$. Hence, all that is needed is that $T^{-1}F_T$ without 
resetting diverges as $T \to\infty$, which is guaranteed by Lemma~\ref{lemma:keyassumption}.
\end{proof}

The interpretation of the above approximation is as follows. The shift of a tiny amount of 
probability mass into the tail of the probability distribution $\mu$ has a negligible cost 
on the exponential scale. The shift produces a small fraction of reset periods that are 
longer than typical. In these reset periods large contributions occur at a negligible cost, 
since the growth without reset is faster than linear. In this way we can produce any $\phi$ 
that is larger than $\phi^{*}_r$ at zero cost on the scale $T$ of the LDP.

\begin{remark}
{\rm Theorem~\ref{thm:genrateprop} captures a potential property of the rate function to 
the right of the mean. A similar property holds to the left of the mean, when $\phi^*_0 = -\infty$ 
and $\phi^*_r > -\infty$ for $r>0$.} 
\end{remark}

%%%

\subsection{Quadratic rate function below the mean}
\label{sec:rfquad}

\begin{theorem}
Suppose that $\phi^{*}_0 = \infty$. For every $r>0$, if $\phi^{*}_r < \infty$, then
\begin{equation}
\chi_r(\phi) \sim C_r(\phi^*_r-\phi)^2, \qquad \phi \uparrow \phi^*_r,
\end{equation}
with $C_r\in (0,\infty)$ a constant that is given by the variational formula in
\eqref{quadratic2}--\eqref{quadratic3} below. (The symbol $\sim$ means 
that the quotient of the left-hand side and the right-hand side tends to $1$.)
\end{theorem}

\begin{proof}
We perturb \eqref{eq:VarRate} around its zero by taking
\begin{equation}
\label{pert}
n= r+m\epsilon, \qquad \mu(\ddd t) = \mathcal{E}_r(\ddd t)\,[1+\nu(t)\epsilon], \qquad
w(t) = u^{*}_r + v(t)\epsilon,
\end{equation}
subject to the constraint $\int_{0}^{\infty} \mathcal{E}_r(\ddd t)\,\nu(t) = 0$, with $\nu(\cdot),
v(\cdot)$ Borel measurable, to ensure that $\mu \in \mathcal{P}([0,\infty))$. This gives
\begin{equation}
I_{r}(r+m\epsilon) = F^*_r(m)\epsilon^{2} + O(\epsilon^{3}), 
\qquad F^*_r(m)=\frac{m^{2}}{2r}.
\end{equation}
Next, we have
\begin{equation}
J_{r}(\mu) = \int_{0}^{\infty} \mathcal{E}_r(\ddd t)\,[1+\nu(t)\epsilon]\log[1+\nu(t)\epsilon].
\end{equation}
Expanding the logarithm in powers of $\epsilon$ and using the normalisation condition, 
we obtain
\begin{equation}
J_{r}(\mu) = G^*_r(\nu)\epsilon^{2} + O(\epsilon^{3}),
\qquad G^*_r(\nu) = \frac{1}{2}\int_{0}^{\infty} \mathcal{E}_r(\ddd t)\,\nu^{2}(t).
\end{equation}
Lastly, we know that (see Fig.~\ref{fig:KL})
\begin{equation}
K_{1}(u^{*}_r+v(t)\epsilon) \sim \tfrac12\,v(t)^{2}\,K''_{1}(u^{*}_r)\epsilon^{2}.
\end{equation}
(As observed below \eqref{KCramer}, $K_{1}$ is strictly convex and smooth on the interior 
of its domain.) Hence the last term in the variational formula becomes
\begin{equation}
\begin{aligned}
&(r+m\epsilon)\int_{0}^{\infty} \mathcal{E}_r(\ddd t)\,[1+\nu(t)\epsilon]\,K_{1}(u^{*}_r+v(t)\epsilon) 
= H^*_r(v)\epsilon^{2} + O(\epsilon^{3}),\\
&H^*_r(v)= \frac{r}{2}K_{1}''(u^{*}_r)\int_{0}^{\infty} \mathcal{E}_r(\ddd t)\,v(t)^{2}.
\end{aligned}
\end{equation}
It follows that
\begin{equation}
\label{quadratic1}
\chi(\phi^{*}_r+\epsilon) = C_r \epsilon^{2} + O(\epsilon^{3})
\end{equation}
with
\begin{equation}
\label{quadratic2}
C_r = \inf_{(m, \nu, v) \in \Phi} \Big\{F^*_r(m)+G^*_r(\nu)+H^*_r(v)\Big\},
\end{equation}
where 
\begin{equation}
\label{quadratic3}
\Phi = \left\{(m,\nu,v)\colon\,
\int_{0}^{\infty} \mathcal{E}_r(\ddd t)\,\nu(t) = 0, \,\, 
r \int_{0}^{\infty} \mathcal{E}_r(\ddd t)\,\left[\frac{m}{r} + \nu(t) + v(t)\right]\,t^{\alpha} = 1\right\}.
\end{equation}
The last constraint guarantees that $n\int_0^\infty \mu(\ddd t)w(t) =\phi^{*}_r+\epsilon+O(\epsilon^{2})$, 
and arises from \eqref{intc1}--\eqref{intc2} after inserting \eqref{pert} and letting $\epsilon \downarrow 0$, 
all for the special case in \eqref{eqdis1}. Finally, it is easy to check that the same argument works 
when \eqref{eqdis1} is replaced by \eqref{eqdis2}. In that case, $t^{\alpha}$ in \eqref{quadratic3} 
becomes $tL(t)$. 

Note that $F^*_r$, $G^*_r$ and $H^*_r$ need not be finite everywhere. However, for the variational formula 
in \eqref{quadratic2} clearly only their finite values matter. Also note that the perturbation is possible only 
for $\epsilon<0$ ($\phi<\phi_r^*$), since there is no minimiser to expand around for $\epsilon>0$ 
($\phi>\phi_r^{*}$), as is seen from Theorem~\ref{thm:genrateprop}.

We have $C_r>0$, because the choice $m=0$, $\nu(\cdot) \equiv 0$, $v(\cdot) \equiv 0$ does not match 
the last constraint. We also have $C_{r}<\infty$, because we can choose $m=r^\alpha/\Gamma(1+\alpha)$,
$\nu(\cdot) \equiv 0$, $v(\cdot) \equiv 0$, which gives $F^*_r(m)=r^{2\alpha-1}/2(\Gamma(1+\alpha))^2$,
$G^*_r(\nu)=0$, $H^*_r(v)=0$. 
\end{proof}

%%%%%%%%%%%%%%%%%%%%%%%%%%%%%%%%%%%%%%%%%%%%%%
\section{Positive occupation time}
\label{sec:occtime}

We now apply the results of Sec.~\ref{sec:genratepropproof} to the three functionals 
of rBM defined in \eqref{eq:ABC}. We start with the positive occupation time, defined as
\begin{equation}
\label{eq:OT}
A_{T} = \int_{0}^{T} 1_{[0,\infty)}(W^r_{t})\, \ddd t.
\end{equation} 
This random variable has a density with respect to the Lebesgue measure, which we 
denote by $p^A_r(a)$, i.e.,
\begin{equation}
p^A_r(a) = \frac{\PP_r( A_T\in \ddd a)}{\ddd a},\qquad a\in (0,T).
\end{equation}
Without resetting, this density is 
\begin{equation}
\label{eq:occtnores}
p^A_0(a) = \frac{1}{\pi\sqrt{a(T-a)}},\qquad a\in (0,T),
\end{equation}
which is the derivative of the famous arcsine law found by L\'evy~\cite{Levy39}. The next 
theorem shows how this result is modified under resetting. 

\begin{theorem}
\label{thm:occtime}
The positive occupation time of rBM has density
\begin{equation}
\label{eq:occtimedistfinal}
p^A_r(a) = \frac{r}{T}\,\eee^{-rT}\,W\big(r\sqrt{a(T-a)}\,\big), \qquad a\in(0,T),
\end{equation}
where
\begin{equation}
\label{eq:scalingfunction}
W(x) = \frac{1}{x}\sum_{j=0}^{\infty}\frac{x^{j}}{\Gamma(\frac{j+1}{2})^2} 
= I_{0}(2x) + \frac{1}{x\pi}\,{}_{1}F_{2}\big(\{1\},\{\tfrac{1}{2}, \tfrac{1}{2}\}, x^{2}\big),
\qquad x\in (0, \infty),
\end{equation}
with $I_{0}(y)$ the modified Bessel function of the first kind with index $0$ and 
${}_{1}F_{2}(\{a\},\{b,c\},y)$ the generalized hypergeometric function 
{\rm \cite[Section 9.6, Formula 15.6.4]{Abramowitz65}}. 
\end{theorem}

\begin{proof}
In what follows, the regions of convergence of the generating functions will be obvious, so 
we do not specify them.
 
The non-reset generating function in \eqref{eq:genfunc} for the occupation time started 
at $X_0=0$ is known to be 
\cite{Majumdar05}
\begin{equation}
\label{eq:satya}
\tG_{0}(k,s) = \frac{1}{\sqrt{s(s-k)}}.
\end{equation}
This can be explicitly inverted to obtain the density in (\ref{eq:occtnores}).

To find the Laplace transform of the reset generating function, we use Theorem~\ref{thm:RR}. 
Inserting \eqref{eq:satya} into \eqref{eq:RR}, we find
\begin{equation}
\label{eq:resetGF}
\tG_{r}(k,s) = \frac{1}{\sqrt{(s+r)(s+r-k)} - r}.
\end{equation}
This can be explicitly inverted to obtain the density in \eqref{eq:occtimedistfinal}, as follows.
Write
\begin{align}
\label{eq:occtimedist}
p^A_r(a) = \eee^{-rT}H(aT,(1-a)T),
\end{align}
where $H$ is to be determined. Substituting this form into \eqref{eq:lapgenfunc}, we get
\begin{equation}
\tG_{r}( k, s) = \int_{0}^{\infty} \ddd T \int_{0}^{1} \ddd a\,\,\eee^{kTa}\,
\eee^{-(s+r)T}\,H(aT,(1-a)T).
\end{equation}
Performing the change of variable $t_{1} = aT$ and $t_{2}=(1-a)T$, we get
\begin{align}
\label{eq:changedvar}
\tG_{r}(k, s) = \int_{0}^{\infty} \ddd t_{1} \int_{0}^{\infty} \ddd t_{2}\,\,
\eee^{-(r+s-k)t_{1}}\eee^{-(r+s)t_{2}}H(t_{1}, t_{2}).
\end{align}
Let $\lambda_{1} = r+s-k$ and $\lambda_{2}= r+s$. Then \eqref{eq:changedvar}, along 
with the right-hand side of \eqref{eq:resetGF}, gives
\begin{align}
\label{eq:occlapbeforeinversion}
\int_{0}^{\infty}\ddd t_{1} \int_{0}^{\infty} \ddd t_{2}\,\,\eee^{-\lambda_{1}t_{1} 
- \lambda_{2}t_{2}}H(t_{1}, t_{2}) = \frac{1}{\sqrt{\lambda_{1}\lambda_{2}} - r}.
\end{align}
To invert the Laplace transform in \eqref{eq:occlapbeforeinversion}, we expand the
right-hand side in $r$,
\begin{equation}
\int_{0}^{\infty}\ddd t_{1} \int_{0}^{\infty} \ddd t_{2}\,\,
\eee^{-\lambda_{1}t_{1} - \lambda_{2}t_{2}}H(t_{1}, t_{2}) 
= \sum_{j=0}^{\infty}\frac{r^{j}}{(\lambda_{1}\lambda_{2})^{(j+1)/2}}, 
\end{equation}
and invert term by term using the identity
\begin{equation}
\frac{1}{\Gamma(\alpha)} \int_{0}^{\infty} \ddd t\,\, t^{\alpha - 1} \eee^{-\lambda t} 
= \frac{1}{\lambda^{\alpha}}, \qquad \alpha>0.
\end{equation}
This leads us to the expression
\begin{equation}
H(t_{1}, t_{2}) = \sum_{j=0}^{\infty}\frac{r^{j}}{\Gamma(\frac{j+1}{2})^{2}}(t_{1}t_{2})^{(j-1)/2} 
= r\sum_{j=0}^{\infty}\frac{(r\sqrt{t_{1}t_{2}})^{j-1}}{\Gamma(\frac{j+1}{2})^{2}}.
\end{equation}
Substituting this expression into \eqref{eq:occtimedist}, we find the result in
\eqref{eq:occtimedistfinal}--\eqref{eq:scalingfunction}.
\end{proof}

The arcsine density in (\ref{eq:occtnores}) is recovered in the limit $r\downarrow 0$ by noting that 
$W(x)\sim (\pi x)^{-1}$ as $x\downarrow 0$. On the other hand, we have
\begin{equation}
W(x)\sim \frac{1}{2\sqrt{\pi x}}\, e^{2x},\qquad x\to\infty
\end{equation}
Consequently,
\begin{equation}
T\,p^A_r(aT) \sim \frac{\sqrt{r}}{2\sqrt{\pi\, T}\, 
\left(a(1-a)\right)^{1/4}}\, \eee^{-r\,T\left(1- 2\sqrt{a(1-a)}\,\right)}, \qquad a \in (0,1),\quad T\to\infty.
\label{sc_large}
\end{equation}
Keeping only the exponential term, we thus find that $(T^{-1}A_T)_{T>0}$ satisfies 
the LDP with speed $T$ and with rate function $\chi^A_r$ given by
\begin{equation}
\label{ldv.rate}
\chi^A_r(a)= r\left(1- 2 \sqrt{a(1-a)}\,\right),\qquad a\in [0,1].
\end{equation}

%%%%%%%%%%%%%%%%%%%%%%%%%%%%%%%%%%%%%%%%%%%%%%
\begin{figure}[t]
\vspace{0.5cm}
\centering
\includegraphics{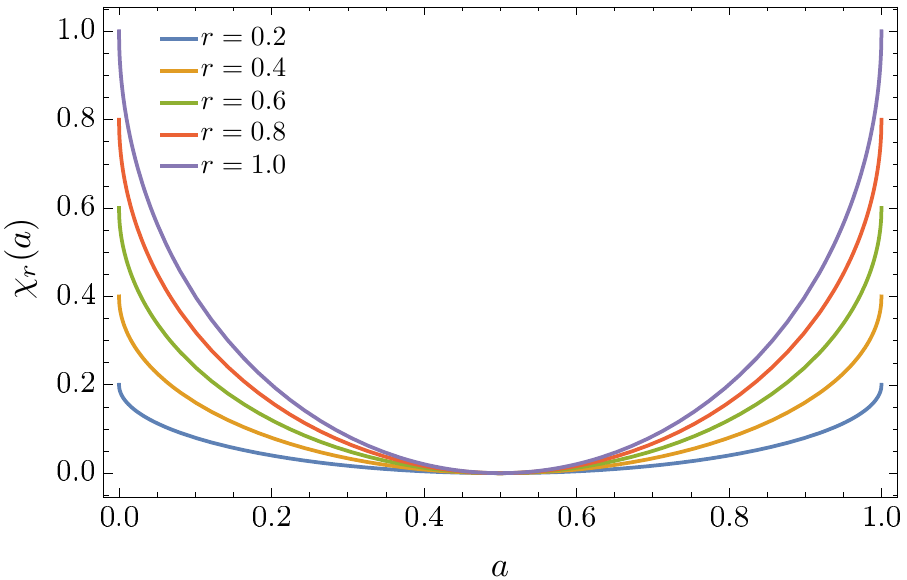}
\caption{Rate function $a \mapsto \chi^A_r(a)$ for the positive occupation time of rBM.}
\label{fig:rateOCC}
\end{figure}
%%%%%%%%%%%%%%%%%%%%%%%%%%%%%%%%%%%%%%%%%%%%%%

The rate function $\chi^A_r$ is plotted in Fig.~\ref{fig:rateOCC}.  As argued in \cite{Meylahn15}, 
\eqref{ldv.rate} can also be obtained by noting that the largest real pole of $\tG(k,s)$ in the 
$s$-complex plane is
\begin{equation}
\lambda_r(k)=\frac{1}{2}\left(k-2r+\sqrt{k^2+4r^2}\right),\qquad k\in\R,
\end{equation}
which defines the scaled cumulant generating function of $A_T$ as $T\to\infty$ (see \eqref{eq:scgfgen1} 
below). Since this function is differentiable for all $k\in\R$, we can use the G\"artner--Ellis Theorem 
\cite[Chapter V]{H00} to identify $\chi^A_r$ as the Legendre transform of $\lambda_r$. 

Note that the positive occupation time does not satisfy the LDP when $r=0$, since $p^A_0(a)$ is not 
exponential in $T$ and does not concentrate as $T\to\infty$. Thus, here resetting is `strong enough' 
to force concentration of $T^{-1}A_T$ on the value $\frac{1}{2}$, with fluctuations around this value 
that are determined by the LDP and the rate function $\chi^A_r$ in \eqref{ldv.rate}. In particular, since 
$\chi^A_r(0)=\chi^A_r(1)=r$, the probability that rBM always stays positive or always stays negative 
is determined on the large deviation scale by the probability $\eee^{-rT}$ of having no reset up to 
time $T$.

Note that $\phi^*_r=\tfrac12$ for $r \geq 0$. Hence the positive occupation time does not satisfy
the condition in Theorem~\ref{thm:genrateprop}. 

%%%%%%%%%%%%%%%%%%%%%%%%%%%%%%%%%%%%%%%%%
\section{Area}
\label{sec:area}

We next consider the area of rBM, defined as 
\begin{equation}
B_{T} =  \int_{0}^{T} W^r_{t}\, \ddd t.
\end{equation}
Its density with respect to the Lebesgue measure is denoted by $p^B_r(b)$, $b\in\R$. The full 
distribution for $T$ fixed is not available, and therefore we start by computing a few moments.

\begin{theorem}
\label{thm:areadist}
For every $T \in (0,\infty)$, the area of rBM for $r>0$ has vanishing odd moments and 
non-vanishing even moments. The first two even moments are
\begin{eqnarray}
\label{eq:areamom}
\EE_{r}[B_{T}^{2}] &=& \frac{2}{r^{3}}\left(rT - 2 + \eee^{-rT}(2+rT)\right),\\ \nonumber
\EE_{r}[B_{T}^{4}] &=& \frac{1}{r^{6}}\left(12(rT)^{2} + 120rT - 840 
+ \eee^{-rT}[9 (rT)^{4} + 68(rT)^{3}+288(rT)^{2}+720\,rT+840]\right).
\end{eqnarray}
\end{theorem}

\begin{proof}
The result follows directly from the renewal formula \eqref{eq:RR} and the Laplace transform of the 
generating function of $B_T$ without resetting,
\begin{align}
\label{eq:Q0}
\tQ_{0}(k, s) = \int_{0}^{\infty} \ddd T \;\eee^{-sT}\mathbb{E}_0[\eee^{kB_{T}}] 
= \int_{0}^{\infty}\ddd T \;\eee^{\frac16 k^{2}T^{3}-sT},
\end{align}
because $B_T$ is a Gaussian random variable with mean 0 and variance $\tfrac13 T^3$. 
Expanding the exponential in $k$ and using \eqref{eq:RR}, we obtain the following expansion 
for the Laplace transform of the characteristic function with resetting:
\begin{equation}
\label{eq:QR}
\tQ_{r}(k, s) = \frac{1}{s} + \frac{1}{s^{2}(r+s)^{2}}k^{2} + \frac{(r+10s)}{s^{3}(r+s)^{5}}k^{4} 
+ O(k^{6}).
\end{equation}
Taking the inverse Laplace transform, we find that the odd moments are all zero, because there are 
no odd powers of $k$, and that the even moments are given by the inverse Laplace transforms 
$\mathcal{L}^{-1}$ of the corresponding even powers of $k$. Thus,
\begin{eqnarray}
\EE_{r}[B_{T}^{2}] &=& \mathcal{L}^{-1}\Big[\frac{2!}{s^{2}(r+s)^{2}}\Big],\nonumber\\
\EE_{r}[B_{T}^{4}] &=& \mathcal{L}^{-1}\Big[\frac{4!(r+10s)}{s^{3}(r+s)^{5}}\Big],
\end{eqnarray}
which yields the results shown in \eqref{eq:areamom}.
\end{proof}

The second moment, which gives the variance, shows that there is a crossover in time from a 
\emph{reset-free regime} characterized by
\begin{equation}
\EE_r[B_T^2]\sim \tfrac13 T^3, \qquad T \downarrow 0,
\end{equation}
which is the variance obtained for $r=0$, to a \emph{reset regime} characterized by
\begin{equation}
\EE_r[B_T^2]\sim \frac{2T}{r^2},\qquad T\to\infty.
\end{equation}
The crossover where the two regimes meet is given by $T=\sqrt{6}/r$, which is proportional 
to the mean reset time. This gives, as illustrated in Fig.~\ref{fig:vararea}, a rough estimate of the time 
needed for the variance to become linear in $T$ because of resetting.

%%%%%%%%%%%%%%%%%%%%%%%%%%%%%%%%%%%%%%%%
\begin{figure}[t]
\vspace{0.3cm}
\includegraphics{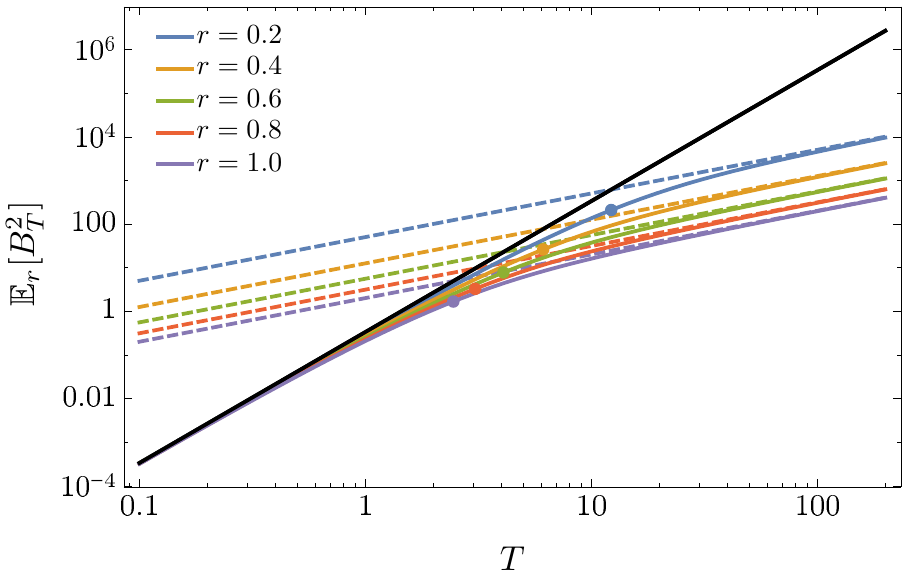}
\caption{Log-log plot of the variance of the area $B_T$ of rBM, showing the crossover from the 
$T^3$-scaling (black line) to the $T$-scaling (dashed lines) for various values of $r$. The filled 
circles show the location of the crossover time $T=\sqrt{6}/r$.}
\label{fig:vararea}
\end{figure}
%%%%%%%%%%%%%%%%%%%%%%%%%%%%%%%%%%%%%%%%

The small fluctuations of $B_T$ of order $\sqrt{T}$ around the origin are Gaussian-distributed. This is 
confirmed by noting that the even moments of $B_T$ scale like
\begin{equation}
\label{eq:dommom1}
\EE_r[B_T^n] \sim \frac{(2n)!}{n!} \left(\frac{\sqrt{T}}{r}\right)^n,\quad T\to\infty,
\end{equation}
so that
\begin{equation}
\EE_r\left[\left(\frac{B_T}{\sqrt{T}}\right)^n\right] \sim \frac{(2n)!}{n! r^n}, \qquad T \to\infty,
\end{equation}
for $n$ even. This implies that the cumulants all asymptotically vanish, except for the variance. 
Indeed, it can be verified that
\begin{equation}
\kappa_{2} = \lim_{T\to\infty} \EE_{r}[T^{-1}B_{T}^{2}] = \frac{2}{r^{2}},
\end{equation}
while
\begin{equation}
\kappa_{4} = \lim_{T\to\infty} \EE_{r}[T^{-2}B_{T}^{4}] - 3\EE_{r}[T^{-1}B_{T}^{2}]^{2}
= \frac{12}{r^{4}} - 3\left(\frac{2}{r^{2}}\right)^{2}=0.
\end{equation}
and similarly for all higher even cumulants. This suggests the following central limit theorem. 

\begin{theorem} 
\label{thm:CLT}
The area of rBM satisfies the central limit theorem,
\begin{equation}
\lim_{T\to\infty} \sigma\sqrt{T}\ p^B_r\left(\frac{b}{\sigma\sqrt{T}}\right) = N(0,1)
\end{equation}
with $N(0,1)$ the standard Gaussian distribution and $\sigma = 2/r^2$.
\end{theorem}

\begin{proof}
We start from the Laplace inversion formula of the renewal formula,
\begin{equation}
p^B_r(b) = \eee^{-rT}\int_{\R}\frac{dk}{2\pi} \eee^{-ikb} \int_{c-i\infty}^{c+i\infty} 
\frac{ds}{2\pi i} \eee^{sT} \,\frac{\tQ_0(k,s)}{1-r \tQ_0(k,s)},
\end{equation}
where $c$ is any value in the region of convergence of $\tQ_0(k,s)$ in the $s$-complex plane. 
Rescaling $b$ by $b=\bar b\sqrt{T}$, as is standard in proofs of the central limit theorem, 
we obtain
\begin{equation}
p^B_r(\bb \sqrt{T}) =\frac{\eee^{-rT}}{\sqrt{T}}\int_{\R}\frac{dl}{2\pi} 
\eee^{-il\bb} \int_{c-i\infty}^{c+i\infty} \frac{ds}{2\pi i}\eee^{sT} \,
\frac{\tQ_0(l/\sqrt{T},s)}{1-r \tQ_0(l/\sqrt{T},s)},
\end{equation}
where $l=k/\sqrt{T}$. Given a fixed $l$ and letting $T\to\infty$, we use the known expression 
of $\EE_0[\eee^{ikB_T}]$ in \eqref{eq:Q0} to Taylor-expand $\tQ_0(k,s)$ around $k=0$,
\begin{equation}
\tQ_0(k,s) = \frac{1}{s}-\frac{k^2}{s^4}+O(k^4),
\end{equation}
to obtain
\begin{equation}
\frac{\tQ_0(l/\sqrt{T},s)}{1-r \tQ_0(l/\sqrt{T},s)}=\frac{1+O(l^2/T)}{s-r+\frac{rl^2}{s^3 T}+O(l^4/T^2)}.
\end{equation}
This expression has a simple pole at
\begin{equation}
s^* = r-\frac{l^2}{r^2 T}+O(l^4/T^2),
\end{equation}
so that, deforming the Bromwich contour through that pole, we get
\begin{equation}
\sqrt{T}\, p^B_r(\bb \sqrt{T}) =\eee^{-rT}\int_{\R}\frac{dl}{2\pi} \eee^{-il\bb} \eee^{s^*T}  
= \int_{\R}\frac{dl}{2\pi}\eee^{-il\bb} \eee^{-l^2/r^2+O(l^4/T)}.
\end{equation} 
As $T\to\infty$, only the quadratic term remains in the exponential, which yields a Gaussian distribution 
with variance $2/r^2$.
\end{proof}

The convergence to the Gaussian distribution can be much slower than the mean reset time, as can 
be seen in Fig.~\ref{fig:vararea}, especially for small reset rates. From simulations, we have found that 
the distribution of $T^{-1/2}B_T$ is well approximated by a Gaussian distribution near the origin. However, 
the tails are strongly non-Gaussian, even for large $T$, indicating that there are important finite-time 
corrections to the central limit theorem, related to rare events involving few resets and, therefore, to 
large Gaussian excursions characterised by the $T^3$-variance.

These corrections can be analysed, in principle, by going beyond the dominant scaling in time of the 
moments shown in \eqref{eq:dommom1}, so as to obtain corrections to the cumulants, which do not 
vanish for finite $T$. It also seems possible to obtain information about the tails by performing a 
saddle-point approximation of the combined Laplace--Fourier inversion formula for values of $B_T$ 
scaling with $T^{3/2}$. We have attempted such an approximation, but have found no results supported 
by numerical simulations performed to estimate $p_r^B(b)$. More work is therefore needed to find 
the tail behavior of this density in the intermediate regime where $T^{1/2} \lesssim b \lesssim T^{3/2}$. 

At this point, we can only establish that $(T^{-1}B_T)_{T>0}$ follows a weak LDP with $\chi_r^B \equiv 0$, implying that
$p_r^B(b)$ decays slower than exponentially on the scale $T$. This follows from the 
general upper bound
\begin{equation}
\label{chirbd}
\chi_r(\phi) \leq \chi_0(\phi)+r \qquad \forall\,\phi \in \R,\,r>0
\end{equation}
found in \cite{Meylahn15a}. We know that $\chi_0^B \equiv 0$, since for every $M \in (0,\infty)$ the probability that the 
Brownian motion stays above $M$ after a time of order $M^2$ decays like $1/\sqrt{T}$ as $T\to\infty$. 
Hence it follows that $\chi^B_r \leq r$. Since rate functions are typically convex, the latter can only mean 
that $\chi_r^B \equiv 0$. 

Note, incidentally, that \eqref{chirbd} is satisfied by the rate function $\chi_r^A$ of the positive occupation 
time (see \eqref{ldv.rate} and Fig.~\ref{fig:rateOCC}).

%%%%%%%%%%%%%%%%%%%%%%%%%%%%%%%%%%%%%%%%
\section{Absolute area}
\label{sec:absolutearea}

We finally consider the absolute area of rBM, defined as
\begin{equation}
C_{T} = \int_{0}^{T} |W^r_{t}|\, \ddd t,
\end{equation}
which can also be seen as the area of an rBM reflected at the origin. Its density with respect
to the Lebesgue measure is denoted by $p^C_r(c)$, $c \in [0,\infty)$. This density was studied 
for pure BM ($r=0$) by Kac \cite{kac1946} and Tak\'acs \cite{Takacs93} (see also \cite{tolmatz2003}).
It satisfies the LDP with speed $T$, when $C_T$ is rescaled by $T$, but with a divergent mean, 
which translates into the rate function tending to zero at infinity (see Fig.~\ref{fig:rrate}). The effect of resetting is to 
bring the mean of $T^{-1}C_T$ to a finite value. Below the mean, we find that the LDP holds with 
speed $T$ and a non-trivial rate function derived from Theorem~\ref{thm:RR}, whereas above 
the mean we find that the rate function vanishes, in agreement with Theorem~\ref{thm:genrateprop}. 
This indicates that the upper tail of $T^{-1}C_T$ decays slower than exponentially in $T$. 

As a prelude, we show how the mean and variance of $C_T$ are affected by resetting. We do not 
know the full distribution, and also the scaling remains elusive. 

\begin{theorem}
\label{thm:absval}
The absolute area of rBM has a mean and a variance given by
\begin{equation}
\EE_{r}[C_{T}] = T^{3/2} f_{1}(rT), 
\qquad \Var_{r}[C_{T}] = T^{3}f_{2}(rT), \qquad r>0,
\end{equation}
where
\begin{equation}
f_{1}(\rho) = \frac{1}{\sqrt{2\pi}} \left[\frac{\eee^{-\rho}}{\rho} 
+ \frac{\sqrt{\pi}}{2 (\rho)^{3/2}}\,(2\rho-1)\,\erf[\sqrt{\rho}\,]\right]
\end{equation}
and
\begin{align}
\label{eq:f2}
f_{2}(\rho) = \frac{1}{8\pi(\rho)^{3}}\left[2\pi\left(2\rho^{2} + \rho- 6 + (5\rho + 6)\eee^{-\rho}\right)
- \left(2\sqrt{\rho}\,\eee^{-\rho} + \sqrt{\pi}(2\rho - 1)\erf[\sqrt{\rho}]\right)^{2}\right].
\end{align}
\end{theorem}

\begin{proof}
The absolute area of pure BM ($r=0$) is known to scale as $T^{3/2}$, so it is convenient to 
rescale $C_T$ as
\begin{equation}
\label{eq:AATrescaled}
C_{T} = T^{3/2} \int_{0}^{1} \ddd t\,|W^r_{t}| = T^{3/2}\,D,
\end{equation}
which defines a new random variable $D$. Expanding \eqref{eq:lapgenfunc} in terms of $k$, 
we get
\begin{align}
\label{eq:G0tildeexpand}
\tG_{0}(k, s) &= \int_{0}^{\infty} \ddd T \eee^{-sT}\left[1 + k T^{3/2}\,\EE_0[D] 
+ \tfrac12k^{2}T^{3}\,\EE_0[D^{2}] + O(k^{3})\right]\nonumber\\
&= \frac{1}{s} + \frac{\EE_0[D]\,\Gamma(\tfrac52)\, k}{s^{5/2}} 
+ \frac{3\,\EE_0[D^{2}]\, k^2}{s^{4}} + O(k^{3}).
\end{align}
Abbreviate $a= \EE_0[D]\,\Gamma(\tfrac52)$ and $b=\EE_0[D^{2}]$ \cite{janson2007}. Inserting 
\eqref{eq:G0tildeexpand} into \eqref{eq:RR}, we find
\begin{equation}
\label{eq:RRexpand}
\tG_{r}(k, s)
= \frac{\frac{1}{s+r} + \frac{ak}{(s+r)^{5/2}} + \frac{3bk^{2}}{(s+r)^{4}} + O(k^{3})}
{1 - r\big[\frac{1}{s+r} + \frac{ak}{(s+r)^{5/2}} + \frac{3bk^{2}}{(s+r)^{4}} + O(k^{3})\big]}
= \frac{\frac{1}{s}\big[1 + \frac{ak}{(s+r)^{3/2}} + \frac{3bk^{2}}{(s+r)^{3}} + O(k^{3})\big]}
{1 - \frac{rak}{s(s+r)^{3/2}} - \frac{3rbk^{2}}{s(s+r)^{3}} + O(k^{3})}.
\end{equation}
Inserting $(1+ck+dk^{2})^{-1} = 1-ck+(c^{2}-d)k^{2} + O(k^{3})$, we obtain 
\begin{equation}
\tG_{r}(k, s) 
= \frac{1}{s} + \frac{a}{s^{2} (s+r)^{1/2}}k + \left(\frac{b}{s^{2}(s+r)^{2}} 
+ \frac{ra^2}{s^3(s+r)^{2}}\right)k^{2} + O(k^{3}).
\end{equation}
We can also expand $\tG_{r}(k, s)$ directly from its definition:
\begin{equation}
\label{eq:lapgenexpand}
\tG_{r}(k, s) 
= \frac{1}{s} + k\int_{0}^{\infty} \ddd T\,\,\eee^{-sT}\, \EE_{r}[C_{T}]  
+ \frac{k^{2}}{2} \int_{0}^{\infty} \ddd T\,\eee^{-sT}\,\EE_{r}[C_{T}^{2}] + O(k^{3}).
\end{equation}
Comparing \eqref{eq:RRexpand} and \eqref{eq:lapgenexpand}, we find
\begin{align}
\int_{0}^{\infty} \ddd T\,\eee^{-sT}\,\EE_{r}[C_{T}]
&= \frac{a}{s^{2} (s+r)^{1/2}},\nonumber\\ 
\frac{1}{2}\int_{0}^{\infty} \ddd T\,\eee^{-sT}\,\EE_{r}[C_{T}^{2}] 
&= \frac{b}{s^{2}(s+r)^{2}} + \frac{ra^2}{s^3(s+r)^{2}}. 
\end{align}

To calculate the first and the second moment, we simply need to invert the Laplace 
transforms. For the mean we find
\begin{equation}
\EE_{r}[C_{T}] = T^{3/2}f_{1}(rT),
\end{equation}
where we use that $\EE_0[D] = \frac{4}{3\sqrt{2\pi}}$ by \cite[Table 3]{Takacs93}. For 
the second moment we use $\EE_0[D^2]=b=\tfrac38$ from the same reference to find
\begin{equation}
\EE_{r}[C_T^{2}] = T^{3}f_{3}(rT)
\end{equation}
with 
\begin{equation}
f_{3}(rT) = \frac{1}{4(rT)^{3}}\Big[2(rT)^{2} + rT -6+(5rT+6)\eee^{-rT}\Big].
\end{equation}
The variance is therefore found to be 
\begin{equation}
\Var_r[C_T] = T^{3}f_3(rT) - T^{3}f_{1}^{2}(rT) = T^{3}f_{2}(rT).
\end{equation}
\end{proof}

The result for the mean converges to $\EE_0[D]$ when $rT\downarrow 0$ and scales like $\tfrac34
\EE_0[D]\sqrt{\frac{\pi}{rT}}$ when $rT \to \infty$. Therefore
\begin{equation}
\lim_{T\to\infty} \EE_{r}[T^{-1}C_{T}] = c^*_r = \frac{1}{\sqrt{2r}}.
\end{equation}
The same analysis for the variance yields
\begin{equation}
\lim_{T\to\infty}T^{-1}\Var_{r}[C_{T}]=\lim_{T\to\infty}T\,\Var_{r}[T^{-1}C_{T}] = \frac{3}{4r^{2}}.
\end{equation}
These two results suggest that $(T^{-1}C_T)_{T>0}$ satisfies the LDP. To compute the corresponding 
rate function, we define the function
\begin{equation}
\label{eq:H}
H(x) = -2^{1/3}\frac{\AI (x)}{\Ai'(x)},
\end{equation}
where 
\begin{equation}
\AI(x) = \int_{x}^{\infty}\text{Ai}(t)\,\ddd t
\end{equation}
is the integral Airy function and $\Ai (x)$ is the Airy function \cite[Section 10.4]{Abramowitz65} 
defined, for example, by 
\begin{equation}
\Ai (x) = \frac{1}{\pi}\int_{0}^{\infty}\cos\left(\tfrac13 t^{3}+xt\right)\,\ddd t.
\end{equation}
The next theorem gives an explicit representation of the rate function of $(T^{-1}C_T)_{T>0}$ for values 
below its mean.

\begin{theorem}
\label{thm:ldpbelow}
Let $c_r^*=1/\sqrt{2r}$, and let $s^*_k$ be the largest real root in $s$ of the equation
\begin{equation}
\label{eq:root1}
\frac{r}{(-k)^{2/3}}H\left(\frac{2^{1/3}(s+r)}{(-k)^{2/3}}\right)=1, \qquad k<0.
\end{equation}
Then $(T^{-1}C_{T})_{T>0}$ satisfies the LDP on $(0,c_r^*)$ with speed $T$ and with rate function 
given by the Legendre transform of $s^*_k$.
\end{theorem}

\begin{proof}
With the same rescaling as in \eqref{eq:AATrescaled}, the generating function for $C_{T}$ 
can be written as 
\begin{equation}
G_{0}(k, T) = \EE_0[\eee^{kT^{3/2}D}]. 
\end{equation}
Using \cite[Eq.~(173)]{janson2007}, we have
\begin{equation}
\int_{0}^{\infty}\eee^{-sT}\EE_0[\eee^{-\sqrt{2}T^{3/2}\xi C_T}]\ddd T 
= - \frac{\text{AI}[\xi^{-2/3}s]}{\xi^{2/3}\text{Ai}'[\xi^{-2/3}s]}, \quad \xi>0,
\end{equation}
so that the Laplace transform of $G_0(k,T)$ has the explicit expression
\begin{equation}
\tG_0(k, s) = \frac{1}{(-k)^{2/3}}H\left(\frac{2^{1/3}s}{(-k)^{2/3}}\right), \quad k<0,
\end{equation}
where $H(x)$ is the function defined in (\ref{eq:H}).

With this result, we follow the method detailed in \cite{Meylahn15}: we insert the expression 
for $\tG_0(k,s)$ into (\ref{eq:RR}) to find the expression for $\tG_r(k,s)$ and locate the largest 
real pole of that function, which is known to determine the \emph{scaled cumulant generating 
function} (SCGF) of $C_T$, defined as 
\be
\lambda_{r}(k) = \lim_{T\to \infty}\frac{1}{T}\log G_{r}(k, T).
\label{eq:scgfgen1}
\ee
Due to the form of $\tG_r(k,s)$ in (\ref{eq:RR}), this pole must be given by the largest real root 
of the equation $r\tG_0(k,s+r)=1$, which yields the equation shown in (\ref{eq:root1}). From there 
we apply the G\"artner--Ellis Theorem \cite{H00} by noting that $\lambda_r(k)=s^*_k$ is finite and 
differentiable for all $k<0$. Consequently, the rate function is given by the Legendre transform 
\begin{equation}
\label{eq:LFtransform}
\chi^C_r(c_k) = k c_k - \lambda_{r}(k),
\end{equation} 
where $c_k = \lambda'_r(k)$ for all $k<0$. It can be verified that $\lambda_r'(k)\to 0$ 
as $k\to -\infty$ and $\lambda_r'(k)\to c_r^*$ as $k\uparrow 0$. Thus, the rate function 
is identified on $(0,c_r^*)$.
\end{proof}

%%%%%%%%%%%%%%%%%%%%%%%%%%%%%%%%%%%%%%%%%%%%%%
\begin{figure}[t]
\centering
\includegraphics{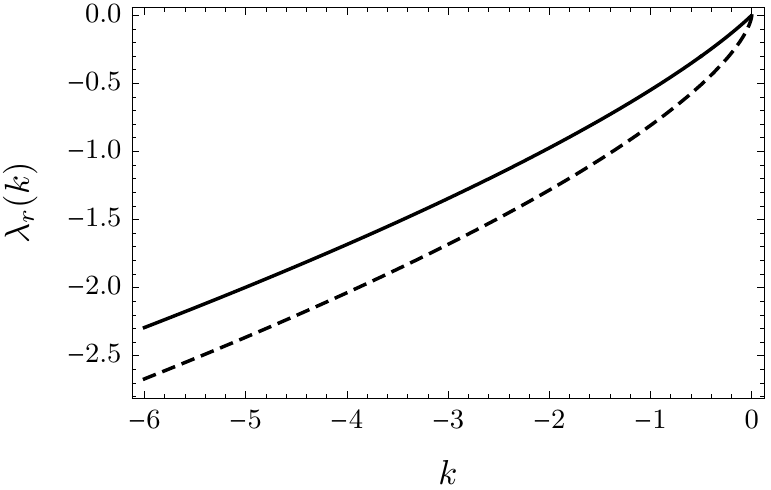}%
\hspace*{0.2in}%
\includegraphics{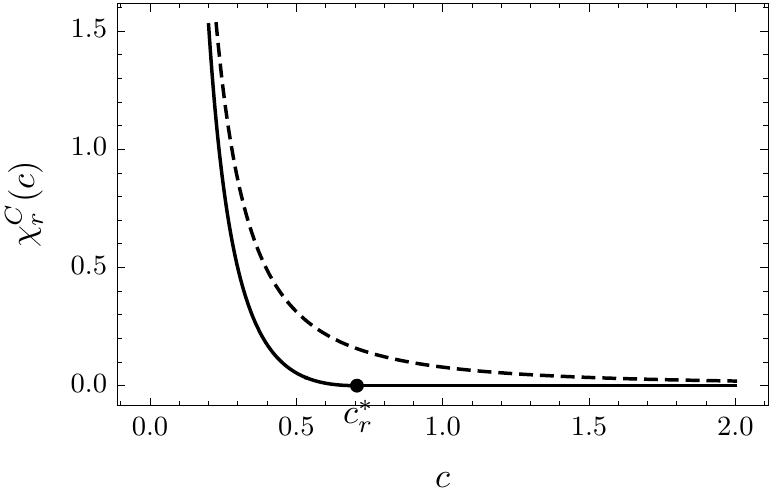}
\caption{Left: SCGF of the absolute area of rBM as a function of $k$ for $r=1$ (full line) and 
$r=0$ (dashed line). Right: Corresponding rate function obtained by Legendre transform for 
$r=1$ (full line) and $r=0$ (dashed line). Above the mean $c_r^*=1/\sqrt{2r}$, $\chi^C_r(c)$ is flat.}
\label{fig:rrate}
\end{figure}
%%%%%%%%%%%%%%%%%%%%%%%%%%%%%%%%%%%%%%%%%%%%%%

The plot on the left in Fig.~\ref{fig:rrate} shows the SCGF $\lambda_r(k)$, while the plot on the right 
shows the rate function $\chi_r^C(c)$ obtained by solving \eqref{eq:root1} numerically and by 
computing the Legendre transform in \eqref{eq:LFtransform}. The rate function is compared with 
the rate function without resetting, which is given by
\begin{equation}
\label{eq:RBMrate}
\chi^C_0(c) = \frac{2|\zeta_{0}'|^{3}}{27\,c^{2}},
\end{equation} 
where $\zeta_{0}'$ is the first zero of the derivative of the Airy function. The derivation of $\chi_0^C$ 
also follows from the G\"artner--Ellis Theorem and is given in Appendix \ref{app:rateAAT}.

Comparing the two rate functions, we see that $T^{-1}C_T$ has a finite mean $c^*_r$ with resetting. 
Above this value, it is not possible to obtain $\chi^C_r(c)$ from $G_r(k,T)$, since the latter function 
is not defined for $k>0$, which indicates that $\chi^C_r(c)$ is either non-convex or has a zero branch 
for $c>c_r^*$ (see \cite[Sec.~4.4]{touchette2009}). Since this is a special case of Theorem~\ref{thm:genrateprop}, 
the second alternative applies, i.e., $\chi^C_r(c)=0$ for all $c>c_r^*$, which implies that the right tail of 
$T^{-1}C_T$ decays slower than $\eee^{-T}$. 

Similar rate functions with zero branches also arise in 
stochastic collision models \cite{lefevere2011b,gradenigo2013}, as well as in non-Markovian random 
walks \cite{harris2009}, and are related to a speed in the LDP that grows slower than $T$. For the 
absolute area of rBM, we do not know what the exact decay of the density of $T^{-1}C_T$ is above 
the mean or whether, in fact, this density satisfies the LDP. This is an open problem.

\section{Conclusion}

In this paper, we have studied the statistical properties of additive functionals of a variant of Brownian motion that is reset at the origin at random intervals,  
and have provided explicit results for three specific functionals, namely, the occupation time, the area, and the 
absolute area. Functionals of standard Brownian motion have been studied extensively in the past, and come with numerous applications in physics and computer science \cite{Majumdar05,Majumdar05a}. In view of these applications, we expect our results for reset Brownian motion to be relevant in a variety of different contexts, in particular, in search-related problems, queuing theory, and population dynamics, which  have all been analysed in the last few years in connection with resetting.

%%%%%%%%%%%%%%%%%%%%%%%%%%%%%%%%%%%%%%%%%%%%%%%%%%
\begin{acknowledgments}
FdH and JMM are supported by NWO Gravitation Grant no 024.002.003-NETWORKS. HT is 
supported by the National Research Foundation of South Africa (Grants 90322 and 96199) 
and Stellenbosch University (Project Funding for New Appointee). The research was supported 
in part by the International Centre for Theoretical Sciences (ICTS), during a visit by the authors 
as participants in the program \emph{Large Deviation Theory in Statistical Physics: Recent 
Advances and Future Challenges} (Code: ICTS/Prog-ldt/2017/8).
\end{acknowledgments}

%%%%%%%%%%%%%%%%%%%%%%%%%%%%%%%%%%%%%%%%%%%%%%%%%%

\appendix

%%%%%%%%%%%%%%%%%%%%%%%%%%%%%%%%%%%%%%%%%

\section{Large deviation principle}
\label{app:LDP}

Let $\mathcal{S}$ be a Polish (i.e., complete separable metric) space. A family $(P_T)_{T >0}$ of 
probability distributions on $\mathcal{S}$ is said to satisfy the \emph{strong} large deviation principle (LDP) with speed 
$T$ and with rate function $I$ when the following three properties hold:
\begin{itemize}
\item[(1)]
$I \not\equiv \infty$. The level sets of $I$, defined by $\{s \in \mathcal{S}\colon\,I(s) \leq c\}$, 
$c \in [0,\infty)$, are compact.
\item[(2)]
$\limsup_{T\to\infty} T^{-1} \log P_T(C) \leq -I(C)$ for all $C\subset \mathcal{S}$ Borel and closed.
\item[(3)]   
$\liminf_{T\to\infty} T^{-1} \log P_T(O) \geq -I(O)$ for all $O\subset \mathcal{S}$ Borel and open.
\end{itemize}
Here
\begin{equation}
I(S) = \inf_{s\in S} I(s), \qquad S \subset \mathcal{S}.
\end{equation}

The family $(P_T)_{T >0}$ is said to satisfy the \emph{weak} LDP when in (1) we only require the level sets 
to be closed and in (2) we only require the upper bound to hold for compact sets. The weak LDP
together with \emph{exponential tightness}, i.e.,
\begin{equation}
\lim_{ {K \uparrow \mathcal{S}} \atop {K \text{ compact}}} 
\limsup_{T\to\infty} T^{-1} \log P_T(\mathcal{S} \setminus K) = -\infty,
\end{equation}
implies the strong LDP. For further background on large deviation theory, the reader is referred 
to \cite[Chapter III]{H00} and \cite{H00,touchette2009}.

%%%%%%%%%%%%%%%%%%%%%%%%%%%%%%%%%%%%%%%%%%%%%%%%%%
\section{Rate function of the absolute area for BM}
\label{app:rateAAT}

The SCGF, defined in (\ref{eq:scgfgen1}), is known to be given for BM 
without resetting by the principal eigenvalue of the following differential operator:
\begin{equation}
\mathcal{L}_{k} = \frac{\sigma^{2}}{2}\frac{d^2}{d x^2} + k |x|,\quad x\in \R,
\end{equation}
called the tilted generator, so that
\begin{equation}
\label{eq:eigqm}
(\mathcal{L}_{k} \psi_{k})(x) = \lambda(k)\psi_{k}(x),
\end{equation}
where $\psi_k(x)$ is the associated eigenfunction satisfying the natural (Dirichlet) boundary conditions 
$\psi(x)\to 0$ as $x\to\pm\infty$ \cite{touchette2017}. Since $|W_t|$ has the same distribution as BM 
reflected at zero, we can also obtain $\lambda(k)$ as the principal eigenvalue of 
\begin{equation}
\mathcal{L}_{k} = \frac{\sigma^{2}}{2}\frac{d^2}{d x^2} + k x,\quad x\geq 0,
\end{equation}
with the Neumann boundary condition $\psi'_k(0)=0$, which accounts for the fact that there is no 
current at the reflecting barrier, in accordance with the Dirichlet boundary condition $\psi_k(\infty)=0$.

The solution $\psi_k(x)$ of both eigenvalue problems is given in terms of the Airy function, $\Ai(\zeta)$,
with
\begin{equation}
\zeta = \Big(\frac{-2k}{\sigma^{2}}\Big)^{1/3} \Big(x - \frac{\lambda(k)}{k}\Big).
\end{equation}
Imposing the boundary conditions, we get a discrete eigenvalue 
spectrum, given by 
\begin{equation}
\lambda^{(i)}(k) = \Big(\frac{\sigma^{2}}{2}\Big)^{1/3}(-k)^{2/3}\zeta'_{i},
\end{equation}
where $\zeta'_{i}$ is the $i$th zero of $\Ai'(x)$. 

The largest eigenvalue $\lambda^{(0)}(k)$ corresponds to the SCGF $\lambda_0(k)$ without resetting 
(see Fig.~\ref{fig:rrate}), which yields the rate function $\chi^C_0$ shown in \eqref{eq:RBMrate}, after 
applying the Legendre transform shown in \eqref{eq:LFtransform}. The function $\lambda_0(k)$ is 
defined only for $k\leq 0$, but since it is steep at $k=0$, the G\"artner--Ellis Theorem can be applied 
in this case.

Note that the spectral method can also be used to find the rate function $\chi^C_r$ of the absolute 
area of rBM, following the method explained in \cite{Meylahn15}. However, the expression for the 
generating function $\tG_0(k,s)$ in this case is explicit, so it is more convenient to use this expression, as is
done in the proof of Theorem~\ref{thm:absval}, in combination with the renewal formula of 
Theorem~\ref{thm:RR}.

%%%%%%%%%%% REFERENCES %%%%%%%%%%%%%%%%%%%%%

%\bibliography{masterbib}

\end{document}